\documentclass[a4paper, 10pt]{article}
\usepackage{amsmath}
\usepackage{amssymb,esint}
\usepackage{amscd}
\usepackage{xspace}
\usepackage{fancyhdr}
\newtheorem{theorem}{Theorem}[section]

\newtheorem{corollary}[theorem]{Corollary}

\newtheorem{lemma}[theorem]{Lemma}

\newtheorem{proposition}[theorem]{Proposition}

\newenvironment{proof}[1][Proof]{\textbf{#1.} }{\hfill\rule{0.5em}{0.5em}}
{\catcode`\@=11\global\let\AddToReset=\@addtoreset
\AddToReset{equation}{section}

\AddToReset{theorem}{section}

\begin{document}
\title{Gradient weighted norm inequalities for very weak solutions of  linear parabolic equations with BMO coefficients}
\author{
  {\bf Quoc-Hung Nguyen\thanks{ E-mail address: quochung.nguyen@sns.it}}\\[0.5mm]
 {\small   Scuola Normale Superiore, Centro Ennio de Giorgi, Piazza dei Cavalieri 3, I-56100
 	Pisa, Italy.}\\}
\date{May 3, 2017}  
\maketitle
\begin{abstract}
In this paper,  we prove the Lorentz space $L^{q,p}$-estimates for gradients of very weak  solutions to the linear parabolic equations with $\mathbf{A}_q$-weights
$$u_t-\operatorname{div}(A(x,t)\nabla u)=\operatorname{div}(F),$$
 in a bounded domain $\Omega\times (0,T)\subset\mathbb{R}^{N+1}$, where $A$ has a small mean oscillation, and $\Omega$ is a Lipchistz domain with a small Lipschitz constant.\medskip
\end{abstract}                          
 \section{Introduction and main results} 
 In this article, we are concerned with the global weighted Lorentz space estimates for gradients of  very weak solutions to linear parabolic equations in divergence form:
 \begin{equation}\label{5hh070120148}
                       \left\{
                                       \begin{array}
                                       [c]{l}%
                                       {u_{t}}-\operatorname{div}(A(x,t)\nabla u)=\operatorname{
                                       div}(F)~~\text{in }\Omega_T,\\ 
                         u=0~~~~~~~\text{on}~~
                                                                                              \partial_p(\Omega \times (0,T)),
                                                                                                 \\                          
                                       \end{array}
                                       \right.  
                                       \end{equation}   where  $\Omega_T:=\Omega\times (0,T)$ is a bounded open subset of $\mathbb{R}^{N+1}$, $N\geq2$, $ \partial_p(\Omega \times (0,T))=(\partial\Omega\times(0,T))\cup (\Omega\times\{t=0\})$,   $F\in L^p(\Omega_T,\mathbb{R}^N),~p>1$ is a given vector field and  the matrix function $A:\mathbb{R}^N\times\mathbb{R}\times \mathbb{R}^N\to \mathbb{R}^N$ is a Carath\'eodory vector valued function, i.e. $A$ is measurable in $(x,t)$ and continuous with respect to $\nabla u$ for a.e. $(x,t)$.\\
                                       We suppose in this paper that $A$ satisfies 
                                       \begin{align}
                                       \label{5hhconda} \Lambda ^{-1}|\xi|^2\leq \langle A(x,t)\xi,\xi\rangle\le \Lambda |\xi|^2,
                                       \end{align}
                                          for every $\xi\in \mathbb{R}^N$ and a.e. $(x,t)\in \mathbb{R}^N\times \mathbb{R}$, where  $\Lambda$ is a  positive constant.
Our main result is that, for any $q>1$ and any $ w\in \mathbf{A}_q$ (the Muckenhoupt class for parabolic, see below), $F\in L^q_w(\Omega_T,\mathbb{R}^N)$, and under some additional conditions on the matrix $A$ and on the boundary of $\Omega$, there exists a unique very weak solution $u\in L^{q_0}(0,T,W_{0}^{1,q_0}(\Omega))$  for some $q_0>1$ of \eqref{5hh070120148} satisfying 
\begin{align}\label{es1}
\int_{\Omega_T}|\nabla u|^q w dxdt\leq C \int_{\Omega_T}|F|^q w dxdt
\end{align}
In this paper, a very weak solution $u$ of \eqref{5hh070120148} is understood in the standard weak (distributional) sense, that is $u\in L^1(0,T,W_0^{1,1}(\Omega))$ is a very weak solution of \eqref{5hh070120148} if 
\begin{align*}
 -\int_{\Omega_T}u \varphi_tdxdt+\int_{\Omega_T}A(x,t)\nabla u\nabla \varphi dxdt= -\int_{\Omega_T}F\nabla\varphi dxdt
 \end{align*}
 for all $\varphi \in C_c^1([0,T)\times \Omega)$.         
                                                             \\          
Case $w\equiv1$, the result was obtained by Byun and Wang in \cite{55BW1,55BW4}. Moreover, $w\in A_{q/2}$ for  $q\geq 2$ and $w^{3}\in \mathbf{A}_1$ for $1<p<2$ was proved by author in \cite[see Theorem 1.3]{55QH3}.  
The result of this paper is inspired by \cite{AdiMenPhuc}, they have demonstrated for linear elliptic equation, their approach employs a local version of the sharp maximal function of Fefferman and Stein. Our approach in this paper is different from \cite{AdiMenPhuc}, we use Hardy-Littlewood maximal function.
 It is worth mentioning that the result of this paper can imply  results in  \cite{AdiMenPhuc}, see Corollary \ref{101120143b}.  Furthermore, the requirement $w\in \mathbf{A}_q$ in \eqref{es1} is optimal, this was discussed in \cite{AdiMenPhuc}. 

      For our purpose, we need to assume that $\Omega$ is a Lipschitz domain with small Lipschiptz constant. We say that $\Omega$ is a $(\delta,R_0)-$Lip domain for $\delta\in (0,1)$ and $R_0>0$ if for every $x\in\partial \Omega$, there exists a  map $\Gamma:\mathbb{R}^{n-1}\to \mathbb{R}$ such that $||\nabla \Gamma||_{L^\infty(\mathbb{R}^{n-1})}\leq \delta$ and, upon rotating and relabeling of coordinates if necessary, 
      \begin{align*}
      \Omega\cap B_{R_0}(x_0)=\{(x',x_n)\in B_{R_0}(x_0): x_n>\Gamma(x')\}.
      \end{align*}
      It is well-known that  $\Omega$ is a $(\delta,R_0)-$Lip domain for $\delta\in (0,1)$ and $R_0>0$ then, $\Omega$ is also  a $(\delta,R_0)-$Reifenberg flat domain, see \cite{55BW1,55BW4,55QH3}. 
     We also require  that the matrix function $A$  satisfies a smallness condition of BMO type in the $x$-variable in the sense that $A(x,t)$ satisfies a $(\delta,R_0)$-BMO condition for some $\delta, R_0>0$ if 
                                   \begin{equation*}
                                   [A]_{R_0}:=\mathop {\sup }\limits_{(y,s)\in \mathbb{R}^N\times\mathbb{R},0<r\leq R_0}\fint_{Q_r(y,s)}|A(x,t)-\overline{A}_{B_r(y)}(t)|dxdt\leq \delta,
                                   \end{equation*}         
 where 
          $\overline{A}_{B_r(y)}(t)$ is denoted the average of $A(t,.)$ over the ball $B_r(y)$, i.e,
                                   \begin{equation*}
                                   \overline{A}_{B_r(y)}(t):=\fint_{B_r(y)}A(x,t)dx.
                                   \end{equation*}                                                              
                                                                      
                                      The above condition appeared in our previous paper \cite{55QH2}. It is easy to  see that the $(\delta,R_0)-$BMO  is satisfied when $A$ is continuous or has small jump discontinuities with respect to $x$.
    We recall that a positive function $w\in L^1_{\text{loc}}(\mathbb{R}^{N+1})$ is called an $\mathbf{A}_{p}$ weight, $1\leq p<\infty$ if there holds
    \begin{align*}
    [w]_{\mathbf{A}_{p}}:= \mathop {\sup }\limits_{\tilde{Q}_\rho(x,t)\subset\mathbb{
    R}^{N+1}}\left(\fint_{\tilde{Q}_\rho(x,t)}w(y,s)dyds
    \right)\left(\fint_{\tilde{Q}_\rho(x,t)}w(y,s)^{-\frac{1}{p-1}}dyds
        \right)^{p-1}<\infty~~\text{when }~p>1,
    \end{align*} 
     The $[w]_{\mathbf{A}_p}$ is called the $\mathbf{A}_{p}$ constant of $w$. \\                                  
   A positive function $w\in L^1_{\text{loc}}(\mathbb{R}^{N+1})$ is called an $\mathbf{A}_{\infty}$ weight if there are two positive constants $C$ and $\nu$ such that
                                $$w(E)\le C \left(\frac{|E|}{|Q|}\right)^\nu w(Q),
                                $$
                                 for all cylinder $Q=\tilde{Q}_\rho(x,t)$ and all measurable subsets $E$ of $Q$. The pair $(C,\nu) $ is called the $\mathbf{A}_\infty$ constant of $w$ and is denoted by $[w]_{\mathbf{A}_\infty}$.  
  It is well known that this class is the union of $\mathbf{A}_p$ for all $p\in (1,\infty)$, see \cite{55Gra}. Furthermore, if $w\in \mathbf{A}_p$ with $[w]_{\mathbf{A}_p}\leq M$ then there exists a constant $\varepsilon_0=\varepsilon(N,p,M)$, and a constant $M_0=M(N,p,M)$ such that $[w]_{\mathbf{A}_{p-\varepsilon_0}}\leq M_0$. 
   If $w$ is a weight function belonging to $w\in \mathbf{A}_{\infty}$ and $E\subset \mathbb{R}^{N+1}$ a Borel set, $0<q<\infty$, $0<p\leq\infty$, the weighted Lorentz space $L^{q,p}_w(E)$  is the set of measurable functions $g$ on $E$ such that 
        \begin{equation*}
     ||g||_{L^{q,p}_w(E)}:=\left\{ \begin{array}{l}
              \left(q\int_{0}^{\infty}\left(\rho^qw\left(\{(x,t)\in E:|g(x,t)|>\rho\}\right)\right)^{\frac{p}{q}}\frac{d\rho}{\rho}\right)^{1/p}<\infty~\text{ if }~s<\infty, \\ 
               \sup_{\rho>0}\rho \left( w\left(\{(x,t)\in E:|g(x,t)|>\rho\}\right)\right)^{1/q}<\infty~~\text{ if }~p=\infty. \\ 
               \end{array} \right.
        \end{equation*}              
               Here we write $w(O)=\int_{O}w(x,t)dxdt$ for a measurable set $O\subset \mathbb{R}^{N+1}$.         Throughout the paper, we always denote  $T_0=\text{diam}(\Omega)+T^{1/2}$ and $Q_\rho(x,t)=B_\rho(x)\times (t-\rho^2,t)$ $\tilde{Q}_\rho(x,t)=B_\rho(x)\times (t-\rho^2/2,t+\rho^2/2)$ for $(x,t)\in\mathbb{R}^{N+1}$ and $\rho>0$.  Moreover,  $\mathcal{M}$ denotes the parabolic Hardy-Littlewood maximal function defined for each locally integrable function  $f$ in $\mathbb{R}^{N+1}$ by
                     \begin{equation*}
                     \mathcal{M}(f)(x,t)=\sup_{\rho>0}\fint_{\tilde{Q}_\rho(x,t)}|f(y,s)|dyds~~\forall (x,t)\in\mathbb{R}^{N+1}.
                     \end{equation*}
            If $q>1$ and $w\in \mathbf{A}_q$ we verify that $\mathcal{M}$ is operator from $L^1(\mathbb{R}^{N+1})$ into $L^{1,\infty}(\mathbb{R}^{N+1})$ and $L^{q,p}_w(\mathbb{R}^{N+1})$  into itself for $0<p\leq \infty$, see  \cite{55Stein2,55Stein3,55Tur}. 
         \\ We would like to mention that the use
            of the Hardy-Littlewood maximal function in non-linear degenerate problems was started in the elliptic setting by T.  Iwaniec in his fundamental paper \cite{Iwa}. \\
                        We now state the main result of the paper.                                      
                \begin{theorem} \label{101120143} For any $w\in \mathbf{A}_{q}$, $1< q<\infty$, $0<p\leq\infty$ we find  $\delta=\delta(N,\Lambda, q,p, [w]_{\mathbf{A}_{q}})\in (0,1)$ such that if $\Omega$ is  $(\delta,R_0)$-Lip domain $\Omega$ and $[A]_{R_0}\le \delta$ for some $R_0>0$ and $F\in L^{q,p}_w(\Omega_T)$, then  there exists a unique weak solution $u\in  L^{q_0}(0,T,W_{0}^{1,q_0}(\Omega))$ for some $q_0>1$ of
                	  \eqref{5hh070120148} satisfying                                    
                  \begin{equation}\label{101120144}
                              |||\nabla u|||_{L^{q,p}_w(\Omega_T)}\leq C ||F||_{L^{q,p}_w(\Omega_T)}.
                                       \end{equation} 
                                        Here $C$ depends only  on $N,\Lambda,q,p, [w]_{\mathbf{A}_q}$ and $T_0/R_0$.                       
                                      \end{theorem}
 As an immediate consequence of Theorem \ref{101120143}, we obtain a version of  Theorem \ref{101120143} for the linear elliptic equations. This result was  obtained in \cite{AdiMenPhuc}.                                                                   
 \begin{corollary}\label{101120143b}  Assume  that $ A(x)=A(x,t)$ for all $(x,t)\in \mathbb{R}^{N+1}$.   For any $w(x)=w(x,t)\in  \mathbf{A}_{q}$, $1< q<\infty$, $0<p\leq\infty$ we find  $\delta=\delta(N,\Lambda, q_0,q,p, [w]_{\mathbf{A}_{q}})\in (0,1)$ such that if $\Omega$ is  $(\delta,R_0)$-Lip domain $\Omega$ and $[A]_{R_0}\le \delta$ for some $R_0>0$ and $G\in L^{q,p}_w(\Omega)$ then there exists a unique very weak solution $u\in W_{0}^{1,q_0}(\Omega)$ for some $q_0>1$ of 
 	\begin{equation}\label{pro-elliptic}
 	\left\{
 	\begin{array}
 	[c]{l}%
 	-\operatorname{div}(A(x)\nabla u)=\operatorname{
 		div}(G)~~\text{in }\Omega,\\ 
 	u=0~~~~~~~\text{on}~~
 	\partial \Omega.
 	\\                          
 	\end{array}
 	\right.  
 	\end{equation}  satisfying                                       
 	\begin{equation}\label{101120144b}
 	|||\nabla u|||_{L^{q,p}_w(\Omega)}\leq C ||F||_{L^{q,p}_w(\Omega)}.
 	\end{equation} 
 	Here $C$ depends only  on $N,\Lambda,q,p, [w]_{\mathbf{A}_q}$ and $diam(\Omega)/R_0$.
 \end{corollary}
    
    \section{Interior estimates and boundary estimates for parabolic equations}
    In this section we present various local interior and boundary estimates for weak solution $u$ of \eqref{5hh070120148}. They will be used for our global estimates later. 
   In \cite{55QH3}, author proved the following result.                                                                                                   
   \begin{theorem} \label{161120141} Let  $q>1$ and $G\in L^{q}(\Omega_T,\mathbb{R}^N)$.  We find a $\delta=\delta(N,\Lambda,q)\in (0,1)$ such that if $\Omega$ is  a $(\delta,R_0)$-Lip domain  and $[A]_{R_0}\le \delta$ for some $R_0>0$ there exists a very unique weak solution $v\in  L^q(0,T,W_0^{1,q}(\Omega))$ of 
   	\begin{equation}\label{161120143}
   	\left\{
   	\begin{array}
   	[c]{l}%
   	{v_{t}}-\operatorname{div}(A(x,t)\nabla u)=\operatorname{div}(G)~~\text{in }\Omega_T,\\ 
   	u=0~~~~~~~\text{on}~~
   	\partial_p(\Omega \times (0,T)),\\                          
   	\end{array}
   	\right.  
   	\end{equation}    
   	Furthermore, there holds                     
   	\begin{equation}\label{161120144}  ||\nabla u||_{L^{q}(\Omega_T)}\leq C ||F||_{L^{q}(\Omega_T)}
   	\end{equation} 
   	where $C$ depends  only on $N,\Lambda,q$ and $(diam(\Omega)+T^{1/2})/R_0$.
   	
   \end{theorem}
Let $s>0$. We apply Theorem \eqref{161120141} to $G=F$and  $q=s$, there is a constant $\delta_0=\delta_0(N,\Lambda,s)\in (0,1/4)$ such that if $\Omega$ is  a $(\delta_0,R_0)$-Lip domain  and $[A]_{R_0}\le \delta_0$ for some $R_0>0$ , then the  problem \eqref{5hh070120148} has a unique very weak solution $u\in  L^{s}(0,T,W_{0}^{1,s}(\Omega))$ satisfying 
\begin{equation*} ||\nabla u||_{L^{s}(\Omega_T)}\leq C ||F||_{L^{s}(\Omega_T)}.
\end{equation*} 
where $C=C(N,\Lambda,s,T_0/R_0)$.\\
In this section, we assume that  $\Omega$ is  a $(\delta_0,R_0)$-Lip domain  and $[A]_{R_0}\le \delta_0$ for some $R_0>0$, where $\delta_0$ is as above.  For some technical reasons, throughout this section, we always assume that  $u\in  L^{s}(-\infty,T;W^{1,s}_0(\Omega))$ is a very weak solution to equation \eqref{5hh070120148} in $\Omega\times (-\infty,T)$ with  $F=0$ in $\Omega\times (-\infty,0)$. \\
      \subsection{Interior Estimates}
Let $R \in (0,R_0)$, $B_{2R}=B_{2R}(x_0)\subset\subset\Omega$ and $t_0\in (0,T)$ . Set $Q_{2R}=B_{2R}  \times (t_0-4R^2,t_0)$ and $\partial_{p}Q_{2R}= \left( {\partial B_{2R}  \times (t_0-4R^2,t_0)} \right) \cup \left( {B_{2R}  \times \left\{ {t = t_0-4R^2} \right\}} \right) $. Since $\Omega$ is  a $(\delta_0,R)$-Lip domain  and $[A]_{R}\le \delta_0$, thus, applying  Theorem \ref{161120141} to $\Omega_T=Q_{2R}$ and $G=F$,   the following equation 
\begin{equation}
\label{111120146}\left\{ \begin{array}{l}
{W_t} - \operatorname{div}\left( {A(x,t)\nabla w} \right) = \operatorname{div}(F) \;in\;Q_{2R}, \\ 
W = 0\quad \quad on~~\partial_{p}Q_{2R}, \\ 
\end{array} \right.
\end{equation}
has a unique very weak solution $W\in  L^s(t_0-4R^2,t_0;W_0^{1,s}(B_{2R}))$.  Moreover, we have 
 \begin{align}\label{1111201410}
\fint_{Q_{2R}}|\nabla W|^{s}dxdt\leq C\fint_{Q_{2R}}|F|^{s}dxdt.
\end{align}
	where $C$ depends  only on $N,\Lambda,s$. Note that the constant $(diam(\Omega)+T^{1/2})/R_0$ in Theorem \ref{161120141}  equals 6 in this case.  \\
 We now set $w=u-W$, so $w$ is a solution of 
        \begin{equation}
       \label{111120146b}
       {w_t} - \operatorname{div}\left( {A(x,t)\nabla w} \right) = 0 \;in\;Q_{2R}
       \end{equation}
       The following a variant of Gehring's lemma was proved in \cite{55Nau,55DuzaMing}. 
       \begin{lemma} \label{111120147}  There exist a  constant $C>0$ depending only on $N,\Lambda$ such that the following estimate      
       \begin{equation}\label{111120148}
       \left(\fint_{Q_{\rho/2}(y,s)}|\nabla w|^{2} dxdt\right)^{\frac{1}{2}}\leq C\fint_{Q_{\rho}(y,s)}|\nabla w| dxdt,
       \end{equation}holds 
               for all  $Q_{\rho}(y,s)\subset \subset Q_{2R}$. 
       \end{lemma} \medskip
      To continue, we denote by $v$ the unique solution
     $v\in L^2(t_0-R^2,t_0;H^1(B_{R}))$
         of the following equation 
         \begin{equation}\label{5hheq4}
         \left\{ \begin{array}{l}
              {v_t} - \operatorname{div}\left( {\overline{A}_{B_R}(t)\nabla v} \right) = 0 \;in\;Q_{R}, \\ 
              v = w\quad \quad on~~\partial_{p}Q_{R}, \\ 
              \end{array} \right.
         \end{equation}
          where $B_R=B_R(x_0)$, $Q_{R}=B_{R}\times(t_0-R^2,t_0)$ and $$\partial_{p}Q_{R}= \left( {\partial B_{R}  \times (t_0-R^2,t_0)} \right) \cup \left( {B_{R}  \times \left\{ {t = t_0-R^2} \right\}} \right).$$.
       \begin{lemma}\label{5hh21101319}  There exist  constants $C_1=C_1(N,\Lambda)$ and  $C_2=C_2(\Lambda)$ such that  \begin{eqnarray}
             \left(  \fint_{Q_R}|\nabla (w-v)|^2dxdt\right)^{1/2}\leq C_1 [A]_{R_0} \fint_{Q_{2R}}|\nabla w|dxdt\, \label{5hh18094}
               \end{eqnarray}
  and \begin{equation}\label{5hh18091}
               C_2^{-1} \int_{Q_R}|\nabla v|^2dxdt\leq  \int_{Q_R}|\nabla w|^2dxdt\leq  C_2\int_{Q_R}|\nabla v|^2dxdt.
                \end{equation}
       \end{lemma}
       \begin{proof} The proof can be found in \cite[Lemma 7.3]{55QH2}. 
       \end{proof}
       \begin{proposition}\label{5hh24092} There holds $$v\in  L^2(t_0-R^2,t_0;H^1(B_{R}))\cap L^\infty(t_0-\frac{1}{4}R^2,t_0;W^{1,\infty}(B_{R/2})), $$ 
       	and
       \begin{equation}\label{5hh18092}
       ||\nabla v||^{s}_{L^\infty(Q_{R/2})}\leq C \fint_{Q_{2R}}|\nabla u|^{s} dxdt +C\fint_{Q_{2R}}|F|^{s} dxdt,
       \end{equation}
         \begin{eqnarray}
               \fint_{Q_R}|\nabla (u-v)|^{s}dxdt\leq C\left([A]_{R_0}\right)^{s}\fint_{Q_{2R}}|\nabla u|^{s}dxdt+ C\fint_{Q_{2R}}|F|^{s} dxdt,\label{5hh18093}
               \end{eqnarray}
               where  $C$ depends  only on $N,\Lambda,s.$\\
       \end{proposition}
       \begin{proof} By standard interior regularity and inequality \eqref{111120148} in Lemma \ref{111120147} and \eqref{5hh18091} in Lemma \ref{5hh21101319} we have\begin{align*}
      ||\nabla v||_{L^\infty(Q_{R/2})}&\leq C \left(\fint_{Q_R}|\nabla v|^2dxdt\right)^{1/2}\\&\leq C \left(\fint_{Q_R}|\nabla w|^2dxdt\right)^{1/2}\\&\leq C \fint_{Q_{2R}}|\nabla w|dxdt.
       \end{align*}   
       Thus, from this and  inequality \eqref{1111201410}, we get \eqref{5hh18092}. \\
        On the other hand, applying \eqref{5hh18094} in Lemma \ref{5hh21101319}  yields  
        \begin{align*}
         \fint_{Q_R}|\nabla (u- v)|^{s}dxdt\leq C \fint_{Q_R}|\nabla (u- w)|^{s}dxdt+ C\left([A]_{R}\right)^{s}\fint_{Q_{2R}}|\nabla w|^{s} dxdt. 
        \end{align*}      
              Combining this with \eqref{1111201410}, we get \eqref{5hh18093}. The proof is complete. 
       \end{proof}  
       \subsection{Boundary Estimates}
       In this subsection, we focus on the corresponding estimates near the boundary. \\
      Throughout this subsection,  $\Omega$ is  a $(\delta/4,R_0)$-Lip domain  and $[A]_{R_0}\le \delta/4$ for  $\delta<\delta_0$. Let $x_0\in \partial\Omega$ be a boundary point and  $0<R<R_0$ and $t_0\in (0,T)$.
      Since, for any $\eta>0$, $B_{\frac{1}{8}}((0,...,\frac{1}{8}-\varepsilon))\cap B_1(0)$  is $(\eta,\eta_0)$-Lip domain for some $\varepsilon>0$ and $\eta_0>0$. Therefore,    there a ball $B$ of radius $R/8$ and $\varepsilon_1,\varepsilon_2>0$ depending only on $N$  such that $B_{\varepsilon_1 R}(x)\subset B\subset B_{R}(x)$ and $B\cap \Omega$ is  $(\delta,\varepsilon_2R)-$ Lip domain. 
      
       We set $\tilde{\Omega}_{R/8}=\tilde{\Omega}_{R/8}(x_0,t_0)=\left(\Omega\cap B \right)\times (t_0-(R/8)^2,t_0)$.  Since  $B\cap \Omega$ is  $(\delta_0,\varepsilon_2R)-$ Lip domain  and $[A]_{\varepsilon_2R}\le \delta_0$, we 
     apply Theorem  \ref{161120141} to $\Omega_T= \tilde{\Omega}_{R/8}$, $G=F$ and $q=s$, there exists a unique  very weak solution $W$ to 
     \begin{equation}\label{161120143b}
     \left\{
     \begin{array}
     [c]{l}%
     {W_{t}}-\operatorname{div}(A(x,t)\nabla W)=\operatorname{div}(F)~~\text{in }\tilde{\Omega}_{R/8},\\ 
     W=0~~~~~~~\text{on}~~
     \partial_p\tilde{\Omega}_{R/8},\\                          
     \end{array}
     \right.  
     \end{equation} 
     satisfying 
     \begin{equation}\label{141120145} ||\nabla W||_{L^{s}(\tilde{\Omega}_{R/8})}\leq C ||F||_{L^{s}(\tilde{\Omega}_{R/8})}
     \end{equation} 
     where $C$ depends  only on $N,\Lambda,s$. Note that the constant $(diam(\Omega)+T^{1/2})/R_0$ in Theorem \ref{161120141}  equals $\frac{3}{8\varepsilon_2}$ in this case.
         In what follows we extend $F$  by zero to $\left(\Omega\times (-\infty,T)\right)^c$, and $W$ by zero to $\mathbb{R}^{N+1}\backslash \tilde{\Omega}_{R/8}$.
         We now set $w=u-W$, so $w$ is a solution of 
       \begin{equation}
         \label{141120142}\left\{ \begin{array}{l}
           {w_t} - \operatorname{div}\left( {A(x,t,\nabla w)} \right) = 0 \;in\;\tilde{\Omega}_{R/8}, \\ 
           w = u\quad \quad on~~\partial_{p}\tilde{\Omega}_{R/8}. \\ 
           \end{array} \right.
         \end{equation}
        
         \begin{lemma}\label{141120141}
               There exist  a constant  $C>0$ depending only on $N,\Lambda$ such that the following estimate    
                \begin{equation}\label{141120143}
                \left(\fint_{Q_{\rho/2}(y,s)}|\nabla w|^{2} dxdt\right)^{\frac{1}{2}}\leq C\fint_{Q_{3\rho}(y,s)}|\nabla w| dxdt,
                \end{equation}
                 holds  for all  $Q_{3\rho}(z,s)\subset B\times (t_0-(R/8)^2,t_0) $.
         \end{lemma}
         Above lemma was proved in \cite[Theorem 7.5]{55QH2}. 
  Next, we set $\rho=\varepsilon_1R(1-\delta)/8$ so that $0<\rho/(1-\delta)<\varepsilon_1R_0/8$. By the definition of Lipschiptz domains and $B_{\varepsilon_1 R}(x_0)\subset B$,  there exists a coordinate system $\{y_1,y_2,...,y_N\}$ with the
  origin $0\in\Omega$ such that in this coordinate system $x_0=(0,...,0,-\frac{\rho\delta}{4(1-\delta)})$ and $B_\rho(0)\subset B$, 
  \begin{equation*}
  B^+_\rho(0)\subset \Omega\cap B_\rho(0)\subset B_\rho(0)\cap \left\{y=(y_1,y_2,....,y_N):y_N>-\frac{\rho\delta}{2(1-\delta)}\right\}.
  \end{equation*}
  Since $\delta<1/4$, we have 
   \begin{equation}\label{5hh1610138}
    B^+_\rho(0)\subset \Omega\cap B_\rho(0)\subset B_\rho(0)\cap \{y=(y_1,y_2,....,y_N):y_N>-\rho\delta\},
    \end{equation}
    where $B^+_\rho(0):=B_\rho(0)\cap\{y=(y_1,y_2,...,y_N):y_N>0\}$.\\
    Furthermore we consider the unique solution
           \begin{equation*}
            v\in  L^2(t_0-\rho^2,t_0;H^1(\Omega\cap B_\rho(0)))
            \end{equation*}
            to the following equation
             \begin{equation}\label{5hh1610131}
             \left\{ \begin{array}{l}
                  {v_t} - \operatorname{div}\left( {\overline{A}_{B_{\rho}(0)}(t)\nabla v} \right) = 0 \;in\;\tilde{\Omega}_\rho(0),\\ 
                  v = w\quad \quad on~~\partial_{p}\tilde{\Omega}_\rho(0), \\ 
                  \end{array} \right.
             \end{equation}
              where $\tilde{\Omega}_\rho(0)=\left(\Omega\cap B_{\rho}(0)\right)\times (t_0-\rho^2,t_0)$.
              We put $v=w$ outside $\tilde{\Omega}_\rho(0)$. As Lemma \ref{5hh21101319} (see \cite[Lemma 2.8]{55QH3}) we have the following result. 
               \begin{lemma}\label{5hh1610139}  There exist positive constants $C_1=C_1(N,\Lambda)$ and  $C_2=C_2(\Lambda)$ such that  \begin{eqnarray}
                             \fint_{Q_{\rho}(0,t_0)}|\nabla (w-v)|^2dxdt\leq C_1 \left([A]_{R}\right)^2 \fint_{Q_{\rho}(0,t_0)}|\nabla w|^2dxdt, \label{5hh1610132}
                             \end{eqnarray}
                    and \begin{equation}\label{5hh1610133}
                             C_2^{-1} \int_{Q_{\rho}(0,t_0)}|\nabla v|^2dxdt\leq  \int_{Q_{\rho}(0,t_0)}|\nabla w|^2dxdt\leq  C_2\int_{Q_{\rho}(0,t_0)}|\nabla v|^2dxdt.
                              \end{equation}
                     \end{lemma}              
                    We can see that if the boundary of $\Omega$ is irregular enough, then the $L^\infty$-norm of $\nabla v$ up to $\partial\Omega\cap B_\rho(0)\times (t_0-\rho^2,t_0)$ may not exist. However, we have the following Lemma obtained in \cite[Lemma 7.12]{55QH2}.                                   
    \begin{lemma}\label{5hh21101314}
        For any $\varepsilon>0$, there exists a small $\delta_1=\delta_1(N,\Lambda,\varepsilon)\in (0,\delta_0)$ such that if  $\delta\in (0,\delta_1)$, there exists a function $V\in C(t_0-\rho^2,t_0;L^2( B_\rho^+(0)))\cap L^2(t_0-\rho^2,t_0;H^1( B_\rho^+(0)))$ satisfying
        \begin{align*}
        ||\nabla V||^2_{L^\infty(Q_{\rho/4}(0,t_0))}\leq C \fint_{Q_{\rho}(0,t_0)} |\nabla v|^2dxdt, 
        \end{align*} 
        and 
 \begin{align*}
  \fint_{Q_{\rho/8}(0,t_0)}|\nabla (v-V)|^2dxdt\leq \varepsilon^2\fint_{Q_{\rho}(0,t_0)} |\nabla v|^2dxdt,
 \end{align*}
    for some $C=C(N,\Lambda)>0$.
        \end{lemma}       
       
\begin{proposition}\label{5hh16101310}For any $\varepsilon>0$ there exists a small $\delta_1=\delta_1(N,\Lambda,s,q_0,\varepsilon)\in (0,\delta_0)$ such that the following holds. If $\Omega$ is a $(\delta/4,R_0)$-Lip domain with $\delta\in (0,\delta_1)$, there is a function $V\in  L^2(t_0-(R/9)^2,t_0;H^1( B_{R/9}(x_0)))\cap L^\infty(t_0-(R/9)^2,t_0;W^{1,\infty}( B_{R/9}(x_0)))$ such that 
 \begin{equation}\label{5hh21101317}
 ||\nabla V||^s_{L^\infty(Q_{\varepsilon_1 R/500})}\leq C\fint_{Q_{R}}|\nabla u|^sdxdt+C \fint_{Q_{R}}|F|^sdxdt,
 \end{equation}
 and 
 \begin{align}
\fint_{Q_{\varepsilon_1 R/500}}|\nabla (u-V)|^sdxdt\leq C (\varepsilon^s+([A]_{R_0})^s)\fint_{Q_{R}}|\nabla u|^sdxdt+ C\fint_{Q_{R}}|F|^sdxdt,\label{5hh21101318}
 \end{align}
 for some $C=C(N,\Lambda,s)>0$. Here $Q_{\rho}=Q_{\rho}(x_0,t_0)$ for all $\rho>0$.
 \end{proposition}
    \begin{proof} We can assume that $\delta\in (0,1/100) $. So
\begin{equation}\label{5hh090520141}
Q_{\varepsilon_1 R/500}\subset Q_{\rho/8}(0,t_0)\subset Q_{6\rho}(0,t_0)\subset Q_{\varepsilon_1 R}\subset Q_{R}
\end{equation}                                                        
    By Lemma  \ref{5hh21101314} for any $\varepsilon>0$, we can find a small positive $\delta=\delta(N,\Lambda,s,q_0,\varepsilon)<1/100$ such that there is a function $V\in  L^2(t_0-\rho^2,t_0;H^1( B_{\rho}(0)))\cap L^\infty(t_0-\rho^2,t_0;W^{1,\infty}( B_{\rho}(0)))$ satisfying 
    \begin{align*}
   &  ||\nabla V||^2_{L^\infty(Q_{\rho/4}(0,t_0))}\leq C \fint_{Q_{\rho}(0,t_0)} |\nabla v|^2dxdt,\\&
   \fint_{Q_{\rho/8}(0,t_0)}|\nabla (v-V)|^2\leq \varepsilon^2\fint_{Q_{\rho}(0,t_0)} |\nabla v|^2dxdt.
    \end{align*}    
        Then, by \eqref{5hh1610133} in Lemma  \ref{5hh1610139} and \eqref{141120143} in Lemma \ref{141120141} and \eqref{5hh090520141},  we get 
        \begin{align}
        \nonumber
                 ||\nabla V||^s_{L^\infty(Q_{\varepsilon_1 R/500})} &\leq C \left(\fint_{Q_{\rho}(0,t_0)} |\nabla w|^2dxdt\right)^{\frac{s}{2}}\nonumber\\&\leq C\fint_{Q_{\varepsilon_1R}} |\nabla w|^sdxdt, \label{5hh21101316}
        \end{align}       
        and 
        \begin{align}
                \fint_{Q_{\varepsilon_1 R/500}}|\nabla (v-V)|^{s}dxdt\leq C \varepsilon^s\fint_{Q_{\varepsilon_1R}} |\nabla w|^sdxdt.\label{5hh21101320}
        \end{align}
        Therefore, from \eqref{141120145} and \eqref{5hh21101316} we get \eqref{5hh21101317}.\\
         Next we prove \eqref{5hh21101318}. Since \eqref{5hh090520141}, 
         \begin{align*}
         &\fint_{Q_{\varepsilon_1 R/500}}|\nabla (u-V)|^sdxdt\leq C\fint_{Q_{\rho/8}(0,t_0)}|\nabla (u-V)|^sdxdt
                    \\&~~~~~~~~\leq C\fint_{Q_{\rho/8}(0,t_0)}|\nabla (u- w)|^sdxdt+C\fint_{Q_{\rho/8}(0,t_0)}|\nabla (w- v)|^sdxdt\\&~~~~~~~~~~+C\fint_{Q_{\rho/8}(0,t_0)}|\nabla (v- V)|^sdxdt.
         \end{align*}
   Using \eqref{141120145} and \eqref{5hh1610132}, \eqref{5hh1610133} in Lemma \ref{5hh1610139} and \eqref{5hh21101320} we find that \begin{align*}
           &\fint_{Q_{\rho/8}(0,t_0)}|\nabla (u-w)|^sdxdt\leq C \fint_{Q_{R}}|F|^sdxdt,
           \\&\fint_{Q_{\rho/8}(0,t_0)}|\nabla( v- w)|^sdxdt \leq C([A]_{R_0})^s \fint_{Q_{\varepsilon_1R}}|\nabla w|^sdxdt\\&~~~~~~~~~~~~~~~~~~~~~~~~~~~~~~~~~~\leq  C([A]_{R_0})^s\left(\fint_{Q_{ R}}|\nabla u|^sdxdt+\fint_{Q_{R}}|F|^sdxdt\right),
           \end{align*}
           and 
           \begin{align*}           
          \fint_{Q_{\rho/8}(0,t_0)}|\nabla (v-V)|^s dxdt \leq C\varepsilon^s\left(\fint_{Q_{ R}}|\nabla u|^sdxdt+\fint_{Q_{R}}|F|^sdxdt\right).
           \end{align*}                                                
           Then we derive \eqref{5hh21101318}.  This completes the proof.

    \end{proof}  
    \section{Global integral gradient bounds for parabolic equations }
    The following good-$\lambda$ type estimate will be essential for our global estimates later. 
  \begin{theorem}\label{5hh23101312} Let $w\in \mathbf{A}_\infty$, $F\in L^s(\Omega_T,\mathbb{R}^N), s>1$.   For any $\varepsilon>0,R_0>0$ one finds  $\delta_1=\delta_1(N,\Lambda,s,\varepsilon,[w]_{\mathbf{A}_\infty})\in (0,1/2)$ and $\delta_2=\delta_2(N,\Lambda,s,\varepsilon,[w]_{\mathbf{A}_\infty},T_0/R_0)\in (0,1)$ and $\Lambda_0=\Lambda_0(N,\Lambda,s)>0$ such that if $\Omega$ is  a $(\delta_1,R_0)$-Lip domain and $[A]_{R_0}\le \delta_1$ then there exists a unique solution $u\in L^s(0,T;W^{1,s}_0(\Omega))$  to equation \eqref{5hh070120148} satisfying 
   \begin{equation}\label{5hh16101311}
   w(\{\mathcal{M}(|\nabla u|^s)>\Lambda_0\lambda, \mathcal{M}(|F|^s)\le \delta_2\lambda \}\cap \Omega_T)\le C\varepsilon w(\{ \mathcal{M}(|\nabla u|^s)> \lambda\}\cap \Omega_T)
   \end{equation}
   for all $\lambda>0$, 
      where the constant $C$  depends only on $N,\Lambda,s, T_0/R_0, [w]_{\mathbf{A}_\infty}$.
  \end{theorem}
  To prove above estimate, we will use  L. Caddarelli and I. Peral's technique in \cite{CaPe}. Namely, it is based on the following technical lemma whose proof is  a consequence of Lebesgue Differentiation Theorem and the standard Vitali covering  lemma, can be found in  \cite{55BW4,55MePh2} with some modifications to fit the setting here.

       \begin{lemma}\label{5hhvitali2} Let $\Omega$ be a $(\delta,R_0)$-Reifenberg flat domain with $\delta<1/4$ and let $w$ be an $\mathbf{A}_\infty$ weight. Suppose that the sequence of balls $\{B_r(y_i)\}_{i=1}^L$ with centers $y_i\in\overline{\Omega}$ and  radius $r\leq R_0/4$ covers $\Omega$. Set $s_i=T-ir^2/2$ for all $i=0,1,...,[\frac{2T}{r^2}]$. Let $E\subset F\subset \Omega_T$ be measurable sets for which there exists $0<\varepsilon<1$ such that  $w(E)<\varepsilon w(\tilde{Q}_r(y_i,s_j))$ for all $i=1,...,L$, $j=0,1,...,[\frac{2T}{r^2}]$; and  for all $(x,t)\in \Omega_T$, $\rho\in (0,2r]$, we have
               $\tilde{Q}_\rho(x,t)\cap \Omega_T\subset F$      
               if $w(E\cap \tilde{Q}_\rho(x,t))\geq \varepsilon w(\tilde{Q}_\rho(x,t))$. Then $
               w(E)\leq \varepsilon Bw(F)$         
               for a constant $B$ depending only on $N$ and $[w]_{\mathbf{A}_\infty}$.
              \end{lemma}   
  \begin{proof}[Proof of Theorem \ref{5hh23101312}] By Theorem \ref{161120141}, we find $\delta_0=\delta_0(N,\Lambda,s)$ then  there exists a unique solution $u$ to solution $u\in L^s(0,T;W^{1,s}_0(\Omega))$  to equation \eqref{5hh070120148} satisfying
  	\begin{align}\label{es-U}
  	\int_{\Omega_T}|\nabla u|^s dxdt\leq C 	\int_{\Omega_T}|F|^s dxdt
  	\end{align}
  	provided that  $\Omega$ is  a $(\delta,R_0)$-Lip domain and $[A]_{R_0}\le \delta$ for $\delta<\delta_0=\delta_0(N,\Lambda,s)$ and $R_0>0$. Let $\varepsilon \in (0,1)$.  Set $$E_{\lambda,\delta_2}=\{\mathcal{M}(|\nabla u|^s)>\Lambda_0\lambda, \mathcal{M}(|F|^s)\leq \delta_2\lambda\}\cap \Omega_T $$ and $$F_\lambda=\{\mathcal{M}(|\nabla u|^s)>\lambda \}\cap \Omega_T$$ for $\delta_2\in (0,1),\Lambda>0$ and $\lambda>0$.
    Let $\{y_i\}_{i=1}^L\subset \Omega$ and a ball $B_0$ with radius $2T_0$ such that 
   $$
    \Omega\subset \bigcup\limits_{i = 1}^L {{B_{r_0}}({y_i})}  \subset {B_0},$$
    where $r_0=\min\{R_0/1000,T_0\}$. Let 
    $s_j=T-jr_0^2/2$ for all $j=0,1,...,[\frac{2T}{r_0^2}]$ and $Q_{2T_0}=B_0\times (T-4T_0^2,T)$. So,
    \begin{equation*}
        \Omega_T\subset \bigcup\limits_{i,j} {{Q_{r_0}}({y_i,s_j})}  \subset {Q_{2T_0}}.
        \end{equation*} 
   We verify that
   \begin{equation}\label{5hh2310131}
   w(E_{\lambda,\delta_2})\leq \varepsilon w({\tilde{Q}_{r_0}}({y_i,s_j})) ~~\forall ~\lambda>0,
   \end{equation}
   for some $\delta_2$ small enough depending on $N,s,\epsilon,[w]_{\mathbf{A}_\infty},T_0/R_0$.\\
   In fact, we can assume that $E_{\lambda,\delta_2}\not=\emptyset$,  so $\int_{\Omega_T}|F|^sdxdt\leq  C|Q_{2T_0}|\delta_2\lambda$. Since $\mathcal{M}$ is a bounded operator from $L^1(\mathbb{R}^{N+1})$ into $L^{1,\infty}(\mathbb{R}^{N+1})$ and \eqref{es-U} we get 
   \begin{align*}
   |E_{\lambda,\delta_2}|& \leq \frac{C}{\Lambda\lambda}\int_{\Omega_T}|\nabla u|^sdxdt\\& \leq 
   \frac{C}{\Lambda\lambda}\int_{\Omega_T}|F|^sdxdt\\&
   \leq C\delta_2 |Q_{2T_0}|,
   \end{align*}
  which implies
 \begin{align*}
    w(E_{\lambda,\delta_2})\leq c\left(\frac{|E_{\lambda,\delta_2}|}{|Q_{2T_0}|}\right)^\nu w(Q_{2T_0})\leq C\delta_2^\nu w(Q_{2T_0}),
   \end{align*}
   where $(c,\nu)=[w]_{\mathbf{A}_\infty}$. It is well-known that (see, e.g \cite{55Gra}) there exist $c_1=c_1(N,c,\nu)$ and $\nu_1=\nu_1(N,c,\nu)$ such that 
   \begin{align*}
   \frac{w(Q_{2T_0})}{w({\tilde{Q}_{r_0}}({y_i,s_j}))}\leq c_1\left(\frac{|Q_{2T_0}|}{|{\tilde{Q}_{r_0}}({y_i,s_j})|}\right)^{\nu_1}~~\forall i,j.
   \end{align*}
  Therefore,
  \begin{align*}
      w(E_{\lambda,\delta_2})\leq C\delta_2^\nu c_1\left(\frac{|Q_{2T_0}|}{|{\tilde{Q}_{r_0}}({y_i,s_j})|}\right)^{\nu_1} w({\tilde{Q}_{r_0}}({y_i,s_j}))
      < \varepsilon w({\tilde{Q}_{r_0}}({y_i,s_j}))~~\forall ~i,j,
     \end{align*}
     for  $\delta_2$ small enough depending on $N,s,\epsilon,[w]_{\mathbf{A}_\infty},T_0/R_0$. Thus \eqref{5hh2310131} follows.\\
  Next we verify that for all $(x,t)\in \Omega_T$, $r\in (0,2r_0]$ and $\lambda>0$ we have
    $
     \tilde{Q}_r(x,t)\cap \Omega_T\subset F_\lambda
     $
     provided $$
        w(E_{\lambda,\delta_2}\cap \tilde{Q}_r(x,t))\geq \varepsilon w(\tilde{Q}_r(x,t)),
      $$
     for some $\delta_2$ small enough depending on $N,s,\epsilon,[w]_{\mathbf{A}_\infty},T_0/R_0$. 
        Indeed,
  take $(x,t)\in \Omega_T$ and $0<r\leq 2r_0$, we set $$\tilde{Q}_\rho=\tilde{Q}_\rho(x,t)~~\forall \rho>0.$$
             Now assume that $\tilde{Q}_r\cap \Omega_T\cap F^c_\lambda\not= \emptyset$ and $E_{\lambda,\delta_2}\cap \tilde{Q}_r\not = \emptyset$ i.e, there exist $(x_1,t_1),(x_2,t_2)\in \tilde{Q}_r\cap \Omega_T$ such that $\mathcal{M}(|\nabla u|^s)(x_1,t_1)\leq \lambda$ and $\mathcal{M}(|F|^s)(x_2,t_2)\le \delta_2 \lambda$.              
              We need to prove that
              \begin{equation}\label{5hh2310133}
                     w(E_{\lambda,\delta_2}\cap \tilde{Q}_r)< \varepsilon w(\tilde{Q}_r). 
                                      \end{equation}
         Using $\mathcal{M}(|\nabla u|^2)(x_1,t_1)\leq \lambda$, we can see that
                                      \begin{equation*}
                                      \mathcal{M}(|\nabla u|^s)(y,t')\leq \max\left\{\mathcal{M}\left(\chi_{\tilde{Q}_{2r}}|\nabla u|^s\right)(y,t'),3^{N+2}\lambda\right\}~~\forall (y,t')\in \tilde{Q}_r.
                                      \end{equation*}
           Therefore, for all $\lambda>0$ and $\Lambda_0\geq 3^{N+2}$,
           \begin{eqnarray}\label{5hh2310134}E_{\lambda,\delta_2}\cap \tilde{Q}_r=\left\{\mathcal{M}\left(\chi_{\tilde{Q}_{2r}}|\nabla u|^s\right)>\Lambda_0\lambda, \mathcal{M}(|F|^s)\leq \delta_2\lambda\right\}\cap \Omega_T \cap \tilde{Q}_r.
           \end{eqnarray}
           In particular, $E_{\lambda,\delta_2}\cap \tilde{Q}_r=\emptyset$ if $\overline{B}_{8r}(x)\subset\subset \mathbb{R}^{N}\backslash \Omega$.
           Thus, it is enough to consider the case $B_{8r}(x)\subset\subset\Omega$ and the case $B_{8r}(x)\cap\Omega\not=\emptyset$.\\   
           First assume  $B_{8r}(x)\subset\subset\Omega$. Let $v$ be as in Theorem \ref{5hh24092} with $Q_{2R}=Q_{8r}(x,t_0)$ and  $t_0=\min\{t+2r^2,T\}$. We have  
           \begin{equation}\label{5hh2310135}
                  ||\nabla v||^s_{L^\infty(Q_{2r}(x,t_0))}\leq C \fint_{Q_{8r}(x,t_0)}|\nabla u|^s dxdt +C\fint_{Q_{8r}(x,t_0)}|F|^s dxdt,
                  \end{equation}
                  and 
                  \begin{align*}
              \fint_{Q_{4r}(x,t_0)}|\nabla (u- v)|^sdxdt\leq  C\fint_{Q_{8r}(x,t_0)}|F|^s dxdt+ C([A]_{R_0})^s\fint_{Q_{8r}(x,t_0)}|\nabla u|^sdxdt.
                  \end{align*}
                  Here constants $C$ in above two  depend only $N,\Lambda,s$. \\
  Thanks to $\mathcal{M}(|\nabla u|^s)(x_1,t_1)\leq \lambda$ and $\mathcal{M}(|F|^s)(x_2,t_2)\le \delta_2 \lambda$ with $(x_1,t_1),(x_2,t_2)\in Q_r(x,t)$, we find $Q_{8r}(x,t_0)\subset\tilde{Q}_{17r}(x_1,t_1),\tilde{Q}_{17r}(x_2,t_2) $ and 
  \begin{align}\nonumber
  ||\nabla v||^s_{L^\infty(Q_{2r}(x,t_0))}&\leq C\fint_{\tilde{Q}_{17r}(x_1,t_1)}|\nabla u|^s dxdt +C\fint_{\tilde{Q}_{17r}(x_2,t_2)}|F|^s dxdt\\&\nonumber\leq
  C(1+\delta_2)\lambda\\&\leq
  C\lambda,\label{1411201410}
  \end{align}
and 
\begin{align}\nonumber
\fint_{Q_{4r}(x,t_0)}|\nabla (u- v)|^2dxdt&\leq  C\delta_2\lambda+ C([A]_{R_0})^s\lambda\\&\leq 
C(\delta_2+\delta_1^s)\lambda.\label{1411201411}
\end{align}                                                   Here we used $[A]_{R_0}\leq \delta_1$ in the last inequality. \\                        
    In view of \eqref{1411201410},  we have that for $\Lambda_0\geq \max\{3^{N+2},2C\}$, $C$ is the constant in \eqref{1411201410}.
 \begin{align*}                           |\{\mathcal{M}\left(\chi_{\tilde{Q}_{2r}}|\nabla v|^s\right)>\Lambda_0\lambda/4\}\cap \tilde{Q}_r|=0.
                           \end{align*}
                  It follows that 
\begin{align*}
|E_{\lambda,\delta_2}\cap \tilde{Q}_r|&\leq   |\{\mathcal{M}\left(\chi_{\tilde{Q}_{2r}}|\nabla (u- v)|^s\right)>\Lambda_0\lambda/4\}\cap \tilde{Q}_r|.                                                 
\end{align*}                           
Therefore, $\mathcal{M}$ is a bounded operator from $L^1(\mathbb{R}^{N+1})$ into $L^{1,\infty}(\mathbb{R}^{N+1})$ and \eqref{1411201411}, $\tilde{Q}_{2r}\subset Q_{4r}(x,t_0)$ we deduce
\begin{align*}
|E_{\lambda,\delta_2}\cap \tilde{Q}_r|&\leq  \frac{C}{\lambda}\int_{\tilde{Q}_{2r}} |\nabla (u-v)|^sdxdt  \\&\leq  C \left(\delta_2+\delta_1^s\right)|\tilde{Q}_r|. 
\end{align*}
Thus,  
\begin{align*}
w(E_{\lambda,\delta_2}\cap \tilde{Q}_r)&\leq c\left(\frac{|E_{\lambda,\delta_2}\cap \tilde{Q}_r |}{|\tilde{Q}_r|}\right)^\nu w(\tilde{Q}_r)
    \\&\leq  C\left(\delta_2+\delta_1^2\right)^\nu w(\tilde{Q}_r)
    \\&< \varepsilon w(\tilde{Q}_r).
\end{align*} 
    where $\delta_2,\delta_1\leq \delta(N,\Lambda,s,\varepsilon,[w]_{\mathbf{A}_\infty})$  and  $(c,\nu)=[w]_{\mathbf{A}_\infty}$.\\
    Next assume $B_{8r}(x)\cap\Omega\not=\emptyset$. Let $x_3\in\partial \Omega$ such that $|x_3-x|=\text{dist}(x,\partial\Omega)$. Set $t_0=\min\{t+2r^2,T\}$. We have 
    
    \begin{equation}\label{5hh2310138}
    Q_{2r}(x,t_0)\subset Q_{10r}(x_3,t_0)\subset Q_{5000r/\varepsilon_1}(x_3,t_0) \subset \tilde{Q}_{10^4r/\varepsilon_1}(x_3,t)\subset \tilde{Q}_{10^5r}(x,t)\subset \tilde{Q}_{10^6r}(x_1,t_1),
    \end{equation}
    and 
    \begin{equation}\label{5hh2310139}
    Q_{5000r/\varepsilon_1}(x_3,t_0) \subset \tilde{Q}_{10^4r/\varepsilon_1}(x_3,t)\subset \tilde{Q}_{10^5r}(x,t)\subset \tilde{Q}_{10^6r}(x_2,t_2)
        \end{equation}
   Applying  Theorem 
    \ref{5hh16101310} with $Q_{R}=Q_{5000r/\varepsilon_1}(x_3,t_0)$ and $\varepsilon=\delta_3\in (0,1)$, there exists a constant $\delta_0'=\delta_0'(N,\Lambda,s,\delta_3)\in (0,\delta_0)$ such that if $\Omega$ is  a $(\delta_0',R_0)$-Lip domain then 
    \begin{equation*}
     ||\nabla V||^s_{L^\infty(Q_{10r}(x_3,t_0))}\leq C\fint_{Q_{5000r/\varepsilon_1}(x_3,t_0)}|\nabla u|^sdxdt+C\fint_{Q_{5000r/\varepsilon_1}(x_3,t_0)}|F|^sdxdt,
     \end{equation*}
     and 
     \begin{align*}
     \nonumber&\fint_{Q_{10r}(x_3,t_0)}|\nabla (u-V)|^sdxdt\\&~~~~\leq C (\delta_3^s+[A]_{R_0}^s)\fint_{Q_{5000r/\varepsilon_1}(x_3,t_0)}|\nabla u|^sdxdt+ C\fint_{Q_{5000r/\varepsilon_1}(x_3,t_0)}|F|^sdxdt.
     \end{align*}      
Since $\mathcal{M}(|\nabla u|^s)(x_1,t_1)\leq \lambda$,  $\mathcal{M}(|F|^s)(x_2,t_2)\le \delta_2 \lambda$ and \eqref{5hh2310138}, \eqref{5hh2310139} we get 
\begin{align}\nonumber
 ||\nabla V||^s_{L^\infty(Q_{10r}(x_3,t_0))}&\nonumber\leq C\fint_{\tilde{Q}_{10^6r}(x_1,t_1)}|\nabla u|^sdxdt+C \fint_{\tilde{Q}_{10^6r}(x_1,t_1)}|F|^sdxdt\\&\nonumber\leq
 C(1+\delta_2)\lambda
 \\&\leq C\lambda,\label{1411201412"}
\end{align}
and 
\begin{align}
     \nonumber&\fint_{Q_{10r}(x_3,t_0)}|\nabla (u-V)|^sdxdt\leq C \left(\delta_3^s+([A]_{R_0})^s+ \delta_2\right)\lambda\\&\leq C\left(\delta_3^s+\delta_1^s+\delta_2\right)\lambda.\label{1411201412}
     \end{align} 
  Notice that we have used  $[A]_{R_0}\leq \delta_1$ in the last inequality.\\  
  As above we also have  that for $\Lambda_0\geq \max\{3^{N+2},4C\}$, the constant $C$ is in \eqref{1411201412"}.
 \begin{align*}
 |E_{\lambda,\delta_2}\cap \tilde{Q}_r|&\leq   |\{\mathcal{M}\left(\chi_{\tilde{Q}_{2r}}|\nabla (u- V)|^s\right)>\Lambda_0\lambda/4\}\cap \tilde{Q}_r|.                         
 \end{align*}     
 Note that the constant $\Lambda_0$ depends only on $N,\Lambda,s$.          \\          
 Therefore using \eqref{1411201412} we obtain 
 \begin{align*}
 |E_{\lambda,\delta_2}\cap \tilde{Q}_r|&\leq  \frac{C}{\lambda}\int_{\tilde{Q}_{2r}} |\nabla (u- V)|^sdxdt \\&\leq  C \left(\delta_3^s+\delta_1^s+ \delta_2\right)|\tilde{Q}_r|. 
 \end{align*}
 Thus
 \begin{align*}
  w(E_{\lambda,\delta_2}\cap \tilde{Q}_r)&\leq c\left(\frac{|E_{\lambda,\delta_2}\cap \tilde{Q}_r|}{|\tilde{Q}_r|}\right)^\nu w(\tilde{Q}_r)
      \\&\leq  C\left(\delta_3^s+\delta_1^s+\delta_2\right)^\nu w(\tilde{Q}_r)
      \\&< \varepsilon w(\tilde{Q}_r),
 \end{align*}   
     where $\delta_1,\delta_2,\delta_3\leq \delta'(N,\Lambda,s,\varepsilon,[w]_{\mathbf{A}_\infty})$ and  $(c,\nu)=[w]_{\mathbf{A}_\infty}$.\\  
    Therefore,  for all $(x,t)\in \Omega_T$, $r\in (0,2r_0]$ and $\lambda>0$, if  
                $$w(E_{\lambda,\delta_2}\cap \tilde{Q}_r(x,t))\geq \varepsilon w(\tilde{Q}_r(x,t)),$$                
then $$ \tilde{Q}_r(x,t)\cap \Omega_T\subset F_\lambda,$$ 
         where $\Omega$ is  a $(\delta_1,R_0)$-Lip domain and  $[A]_{R_0}\le \delta_1$ with  $\delta_1=\delta_1(N,\Lambda,s,\varepsilon,[w]_{\mathbf{A}_\infty})\in (0,\delta_0)$, $\delta_2=\delta_2(N,\Lambda,s,\varepsilon,[w]_{\mathbf{A}_\infty},T_0/R_0)\in (0,1)$. Hence, combining this with \eqref{5hh2310131}, we can apply Lemma \ref{5hhvitali2} to get the result.
              \end{proof}\medskip\\
\begin{proof}[Proof of Theorem \ref{101120143}] Let $F\in L^{q,p}_w(\Omega_T)$. Since $L^{q,p}_w(\Omega_T)\subset L^{q_0}(\Omega_T)$ for some $q_0>1$.  By Theorem \ref{161120141}, we find $\delta_0=\delta_0(N,\Lambda,q,p)$ then  there exists a unique very weak solution $u$ to solution $u\in L^{q_0}(0,T;W^{1,q_0}_0(\Omega))$  to equation \eqref{5hh070120148} satisfying
	\begin{align*}
	\int_{\Omega_T}|\nabla u|^{q_0} dxdt\leq C 	\int_{\Omega_T}|F|^{q_0} dxdt
	\end{align*}
	provided that  $\Omega$ is  a $(\delta,R_0)$-Lip domain and $[A]_{R_0}\le \delta$ for $\delta<\delta_0=\delta_0(N,\Lambda,q,p)$.\\
	By Theorem \ref{5hh23101312}, for any $s\in (0,q_0), \varepsilon>0,R_0>0$ one finds  $\delta=\delta(N,\Lambda,\varepsilon,s,[w]_{\mathbf{A}_\infty})\in (0,\delta_0)$ and $\delta_2=\delta_2(N,\Lambda,\varepsilon,s,[w]_{\mathbf{A}_\infty},T_0/R_0)\in (0,1)$ and $\Lambda_0=\Lambda_0(N,\Lambda)>0$ such that if $\Omega$ is  a $(\delta,R_0)$- Lip domain and $[A]_{R_0}\le \delta$ then 
   \begin{equation}\label{1411201413}
   w(\{\mathcal{M}(|\nabla u|^s)>\Lambda_0\lambda, \mathcal{M}[|F|^s]\le \delta_2\lambda \}\cap \Omega_T)\le C\varepsilon w(\{ \mathcal{M}(|\nabla u|^s)> \lambda\}\cap \Omega_T),
   \end{equation}
   for all $\lambda>0$, 
      where the constant $C$  depends only on $N,\Lambda,s, T_0/R_0, [w]_{\mathbf{A}_\infty}$.
      Thus, for $s<\infty,$
      \begin{align*}
& ||\mathcal{M}(|\nabla u|^s)||_{L^{q_1,p_1}_w(\Omega_T)}^{p_1}
 =q_1\Lambda_0^{p_1}\int_{0}^{\infty}\lambda^{p_1}\left(w(\{\mathcal{M}(|\nabla u|^s)>\Lambda_0\lambda \}\cap \Omega_T)\right)^{p_1/q_1}\frac{d\lambda}{\lambda} \\&~~~\leq 
q_1\Lambda_0^{p_1}2^{p_1/q_1}(C\varepsilon)^{p_1/q_1}\int_{0}^{\infty}\lambda^{p_1}\left(w(\{\mathcal{M}(|\nabla u|^s)>\lambda \}\cap \Omega_T)\right)^{p_1/q_1}\frac{d\lambda}{\lambda} 
\\&~~~~+  q_1\Lambda_0^{p_1}2^{p_1/q_1}\int_{0}^{\infty}\lambda^{p_1}\left(w(\{\mathcal{M}(|F|^s)>\delta_2\lambda \}\cap \Omega_T)\right)^{p_1/q_1}\frac{d\lambda}{\lambda}
\\&~~~ = \Lambda_0^{p_1}2^{p_1/q_1}(C\varepsilon)^{p_1/q_1}||\mathcal{M}(|\nabla u|^s)||_{L^{q_1,p_1}_w(\Omega_T)}^{p_1}+\Lambda_0^{p_1}2^{p_1/q_1}\delta_2^{-p_1}||\mathcal{M}(|F|^s)||_{L^{q_1,p_1}_w(\Omega_T)}^{p_1}.
      \end{align*}
It implies
\begin{align*}
||\mathcal{M}(|\nabla u|^s)||_{L^{q_1,p_1}_w(\Omega_T)}&\leq 2^{1/p_1}\Lambda_02^{1/q_1}(C\varepsilon)^{1/q_1}||\mathcal{M}(|\nabla u|^2)||_{L^{q_1,p_1}_w(\Omega_T)}\\&~~~+2^{1/p_1}\Lambda_0 2^{1/q_1}\delta_2^{-1}||\mathcal{M}(|F|^s)||_{L^{q_1,p_1}_w(\Omega_T)}
\end{align*}
and this inequalities  is also true when $p_1=\infty$. \\
We can choose $\varepsilon=\varepsilon(N,\Lambda,q_1,p_1,C)>0$ such that   $2^{1/p_1}\Lambda2^{1/q_1}(C\varepsilon)^{1/q_1}\leq 1/2$, then we get 
\begin{align}\label{es-U'}
||\mathcal{M}(|\nabla u|^s)||_{L^{q_1,p_1}_w(\Omega_T)}\leq C||\mathcal{M}(|F|^s)||_{L^{q_1,p_1}_w(\Omega_T)}
\end{align}
Let $w\in \mathbf{A}_q $, there exists $q_2=q_2(N,q,[w]_{\mathbf{A}_q})\in (1,q)$ such that $[w]_{\mathbf{A}_{q_2}}\leq C_0=C_0(N,q, [w]_{\mathbf{A}_q})$.Thus, 
\begin{align*}
||\mathcal{M}(|F|^{q/q_2})||_{L^{q_2,pq_2}_w(\Omega_T)}\leq C ||F||_{L^{q,p}_w(\Omega_T)}.
\end{align*} Applying \eqref{es-U'} to $s=q/q_2$ and $q=q_2,p_1=pq_2$, we have 
\begin{align*}
|||\nabla u|||_{L^{q,p}_w(\Omega_T)}&\leq ||\mathcal{M}(|\nabla u|^{q/q_2})||_{L^{q_2,pq_2}_w(\Omega_T)}\\&\leq C||\mathcal{M}(|F|^{q/q_2})||_{L^{q_2,pq_2}_w(\Omega_T)}
\\&\leq C ||F||_{L^{q,p}_w(\Omega_T)}.
\end{align*}
We get the result. The proof is complete.
\end{proof}                                                   


\begin{thebibliography}{99}
	\bibitem{AdiMenPhuc} K. Adimurthi, T. Mengesha, N. C. Phuc, {\em Gradient weighted norm inequalities for linear elliptic equations with discontinuous coefficients}
\bibitem{55BW1}  S.S. Byun, L. Wang, {\em Parabolic equations with BMO nonlinearity in Reifenberg domains}, J. Reine Angew. Math. {\bf 615}, 1-24  (2008). 
\bibitem{55BW4}  S.S. Byun, L. Wang, {\em Parabolic equations in Reifenberg domains}, Arch. Ration. Mech. Anal. {\bf 176}(2), 271-301 (2005).
\bibitem{CaPe} L. Caffarelli, I. Peral: On $W^{1,p}$ {\em  estimates for elliptic equations in divergence form}, Comm. Pure Appl. Math. {\bf 51}, 1-21 (1998).

\bibitem{55DuzaMing} F. Duzaar, G. Mingione, {\em Gradient estimate via non-linear potentials},  Amer. J. Math., {\bf133}(4), 1093-1149  (2011). 
\bibitem{55Gra} L. Grafakos, {\em Classical and Modern Fourier Analysis}, Pearson Education, Inc., Upper Saddle River, NJ, (2004), xii+931 pp.
\bibitem{Iwa} T. Iwaniec, {\em  Projections onto gradient fields and L p-estimates for degenerated elliptic operators}, Studia
Math. {\bf 75}, 293-312 (1983).
\bibitem{55MePh2} T. Mengesha,  N. C. Phuc, {\em Weighted and regularity estimates for nonlinear equations on Reifenberg flat domains},  J. Differ. Equa.  {\bf250}(5), 1485-2507  (2011).
\bibitem{55Nau} J. Naumann and J. Wolf, {\em Interior integral estimates on weak solutions of nonlinear parabolic systems.} Inst. fur Math., Humboldt Universitet, Bonn (1994).
 \bibitem{55QH2} Quoc-Hung Nguyen; {\em Potential estimates and quasilinear parabolic equations with measure data}, Submitted for publication.
\bibitem{55QH3} Quoc-Hung Nguyen {\em Global estimates for quasilinear parabolic equations on Reifenberg flat domains and its applications to  Riccati type parabolic equations with distributional data}, Calculus of Variations and Partial Differential Equations, {\bf  54},  3927-3948 (2015).


\bibitem{55Stein2} E.M. Stein, {\em Singular integrals and differentiability properties of functions}, Princeton
Mathematical Series, {\bf 30}, Princeton University Press, Princeton, (1970).
\bibitem{55Stein3} E.M. Stein, {\em Harmonic analysis: real-variable methods, orthogonality, and oscillatory integral},  {\bf 43}, Princeton University Press, Princeton, (1993).


\bibitem{55Tur} B. O. Tureson, {\em Nonlinear Potential Theory and weighted Sobolev Spaces}, Lecture Notes in Mathematics, {\bf 1736}, Springer-Verlag (2000).

\end{thebibliography}
\end{document}